\documentclass{amsart}

\usepackage{graphicx}
\usepackage{amsmath}
\usepackage{amscd}
\usepackage{amssymb}
\usepackage[all]{xy}

\newtheorem{theorem}{Theorem}[section]

\newtheorem{corollary}[theorem]{Corollary}
\newtheorem{lemma}[theorem]{Lemma}

\theoremstyle{definition}
\newtheorem{definition}[theorem]{Definition}
\theoremstyle{remark}
\newtheorem{remark}[theorem]{Remark}

\numberwithin{equation}{section}

\begin{document}


\title[Lefschetz fibration structures on knot surgery
       $4$-manifolds]{Lefschetz fibration structures on knot surgery $4$-manifolds}

\author{Jongil Park}
\address{Department of mathematical sciences,
  Seoul National University, 599 Gwanak-ro, Gwanak-gu, Seoul 151-747,
  Republic of Korea}
\email{jipark@snu.ac.kr}

\author{ Ki-Heon Yun}
\address{Department of Mathematics, Sungshin Women's University,
  54-7 Dongseondong-gil, Seongbuk-gu, Seoul 136-742, Republic of Korea}
\email{kyun@sungshin.ac.kr}%



\date{June 27, 2009}

\begin{abstract}
 In this article we study Lefschetz fibration structures on knot surgery $4$-manifolds obtained from an elliptic surface $E(2)$ using Kanenobu knots $K$. As a result, we get an infinite family of simply connected mutually diffeomorphic $4$-manifolds coming from  a pair of inequivalent Kanenobu knots. We also obtain an infinite family of simply connected symplectic $4$-manifolds, each of which admits more than one inequivalent Lefschetz fibration structures of the same generic fiber.
\end{abstract}

\maketitle

\section{Introduction}
\label{section:intro}

 Since Seiberg-Witten theory was introduced in 1994,
 many techniques in $4$-dimensional topology have been developed to show
 that a large class of simply connected smooth $4$-manifolds admit
 infinitely many distinct smooth structures.
 Among them, a knot surgery technique introduced by R.~Fintushel and
 R.~Stern turned out to be one of the most powerful tools changing
 the smooth structure on a given $4$-manifold~\cite{FS:98}.
 The knot surgery construction is  following:  Suppose that
 $X$ is a simply connected smooth $4$-manifold containing an embedded
 torus $T$ of square $0$. Then, for any knot $K \subset S^3$,
 one can construct a new $4$-manifold,
 called {\em a knot surgery $4$-manifold},
 \begin{equation*}
  X_K = X\sharp_{T=T_m} (S^1\times M_K)
 \end{equation*}
 by taking a fiber sum along a torus $T$ in $X$ and $T_m = S^1\times m$ in
 \mbox{$S^1 \times M_K$}, where $M_K$ is the $3$-manifold obtained by doing
 $0$-framed surgery along $K$ and $m$ is the meridian of $K$.
 Then Fintushel and Stern proved that, under a mild condition on $X$
 and $T$, the knot surgery $4$-manifold $X_K$ is homeomorphic,
 but not diffeomorphic, to a given $X$~\cite{FS:98}.
 Furthermore, if $X$ is a simply connected elliptic surface $E(2)$,
 $T$ is the elliptic fiber, and $K$ is a fibred knot,
 then it is also known that the knot surgery $4$-manifold $E(2)_K$ admits
 not only a symplectic structure but also a genus $2g(K)+1$ Lefschetz
 fibration structure~\cite{FS:2004, Yun:08}.
 Note that there are only two inequivalent genus one fibred knots, but there are infinitely many inequivalent genus $g$ fibred knots for $g \ge 2$. So one may dig out some interesting properties of $E(2)_K$ by carefully investigating genus two fibred knots and related Lefschetz fibration structures.

 On the one hand, Fintushel and Stern~\cite{FS_ICM98} conjectured that the set of all knot surgery $4$-manifolds of the form $E(2)_K$ up to diffeomorphism is one-to-one correspondence with the set of all knots in $S^3$ up to knot equivalence. Some progresses related to the conjecture were obtained by S. Akbulut~\cite{Akbulut:02} and M. Akaho~\cite{Akaho:06}. But a complete answer to the conjecture for prime knots up to mirror image is not known yet.
 Furthermore, Fintushel and Stern~\cite{FS:2004} also questioned whether any two in the following $4$-manifolds
 \[
  \{ Y(2; K_1, K_2):= E(2)_{K_1} \sharp_{id:\Sigma_{2g+1} \to \Sigma_{2g+1}} E(2)_{K_2} \ |\ K_1, K_2 \ \mathrm{are\ genus\ g\ fibred\ knots} \}
\]
 are mutually diffeomorphic or not. The second author obtained a partial result related to this question under the constraint that one of $K_i$ $(i\!=\!1,2)$ is fixed~\cite{Yun:08}.

 In this article we investigate Lefschetz fibration structures on the knot surgery $4$-manifold $E(2)_K$, where $K$ ranges a family of Kanenobu knots.
 Remind that Kanenobu~\cite{Kanenobu:86,Kanenobu:86a} found an interesting family of inequivalent genus $2$ fibred prime knots
\[
 \{ K_{p,q} \ |\ (p,q) \in \mathcal{R} \ \} \textrm{\ and \ } \mathcal{R}
 =\{ (p,q) \in \mathbb{Z}^2\ | \ p \in \mathbb{Z}^+, \ -p \le q \le p \},
\]
 where any two of them are not in mirror relation and all of them have the same Alexander polynomials.
 In Section~\ref{section:isom} we consider the following family of simply connected symplectic $4$-manifolds which have the same Seiberg-Witten invariants
\[
  \{ Y(2; K_{p,q}, K_{r,s}):= E(2)_{K_{p,q}} \sharp_{id:\Sigma_5 \to \Sigma_5} E(2)_{K_{r,s}} \ |\ (p,q), (r,s) \in \mathcal{R} \}.
\]
 By investigating the monodromy factorization expression corresponding to Lefschetz fibration structure on $Y(2; K_{p,q}, K_{r,s})$, we answer the question raised in~\cite{FS:2004}.

\begin{theorem}
\label{mainthm-1}
 Any two simply connected symplectic $4$-manifolds in
 \[
  \{ Y(2; K_{p,q}, K_{p+1,q})\ | \ p, q \in \mathbb{Z} \}
  \cup \{ Y(2;K_{p,q}, K_{p,q+1}) \ | \ p, q \in \mathbb{Z} \}
 \]
 are mutually diffeomorphic.
\end{theorem}

 In section~\ref{section:nonisom} we also study nonisomorphic Lefschetz fibration structures on simply connected symplectic $4$-manifolds which
  share the same Seiberg-Witten invariants. Let $\xi_{p,q}$ be a genus five
 Lefschetz fibration structure on $E(2)_{K_{p,q}}$.
 Then, by investigating  the monodromy group $G_F(\xi_{p,q})$ of $\xi_{p,q}$, we get the following theorem.

\begin{theorem}
\label{mainthm-2}
 $\xi_{p,q}$ is not equivalent to $\xi_{r,s}$ if $(p,q) \not\equiv (r,s) \pmod 2$.
\end{theorem}

 As a corollary, we can easily recapture a similar result in~\cite{Park-Yun:08}. Remind that we constructed a pair of nonisomorphic Lefschetz fibration structures on $E(n)_K$ for a special type of $2$-bridge knot $K$ in~\cite{Park-Yun:08}. Theorem~\ref{mainthm-2} above also confirms such a phenomena.
 That is, for any $(p, q) \in \mathbb{Z}^2$ with $p \not\equiv q$ (mod $2$),
 $K_{p,q}$ is equivalent to $K_{q,p}$ and therefore $E(2)_{K_{p,q}}$ is the same symplectic $4$-manifold as $E(2)_{K_{q,p}}$. But the theorem above implies that the corresponding Lefschetz fibration structures $\xi_{p,q}$ and $\xi_{q,p}$ are not equivalent.

\subsection*{Acknowledgment}
 Jongil Park holds a joint appointment at KIAS and
 in the Research Institute of Mathematics, SNU. Ki-Heon Yun was supported by Sungshin Women's University Research Grant of 2008.


\bigskip

\section{Preliminaries}\label{section:prelim}

 In this section we briefly review some well-known facts about Lefschetz fibrations on $4$-manifolds and surface mapping class groups (refer to~\cite{GS:99} for details).

\begin{definition}
 \label{defn:lefschetz}
 Let $X$ be a compact, oriented smooth 4-manifold. A
 Lefschetz fibration is a proper smooth map $\pi : X \to B$, where $B$
 is a compact connected oriented surface and $\pi^{-1}(\partial B) =
 \partial X$ such that
\begin{itemize}
 \item[(1)]  the set of critical points $C= \{p_1, p_2, \cdots, p_n\}$ of $\pi$
        is non-empty and lies in $int(X)$ and $\pi$ is injective on $C$
 \item[(2)]  for each $p_i$ and $b_i:=\pi(p_i)$, there are local complex
       coordinate charts agreeing with the orientations of $X$ and $B$ such that
        $\pi$ can be expressed as $\pi(z_1, z_2) = z_1^2 + z_2^2$.
\end{itemize}
\end{definition}

 It is known that there is one-to-one correspondence between the set of symplectic Lefschetz fibrations over $S^2$ and the set of factorizations of the identity elements in the mapping class group as a product of right-handed Dehn twists up to Hurwitz moves and global conjugation~\cite{GS:99, Kas:80, Matsumoto:96}. Therefore a monodromy factorization of a given Lefschetz fibration has lots of information about the underlying symplectic $4$-manifold.

 Two monodromy factorizations $W_1$ and $W_2$ are called \emph{Hurwitz equivalence} if $W_1$ can be changed to $W_2$ in finitely many steps of the following two operations:
\begin{itemize}
 \item[(1)] \emph{Hurwitz move:}
       $t_{c_n}\cdot ... \cdot t_{c_{i+1}} \cdot t_{c_i}
       \cdot ... \cdot t_{c_1} \sim  t_{c_n}\cdot ... \cdot t_{c_{i+1}}(t_{c_i})
       \cdot t_{ c_{i+1}} \cdot ... \cdot t_{c_1}$
 \item[(2)] \emph{inverse Hurwitz move:}
       $t_{c_n}\cdot ... \cdot t_{c_{i+1}} \cdot
       t_{ c_i} \cdot ... \cdot t_{c_1} \sim  t_{c_n}\cdot ...  \cdot t_{ c_{i}}
       \cdot t_{c_{i}}^{-1}(t_{c_{i+1}}) \cdot ... \cdot t_{c_1}$
\end{itemize}
 where $t_a(t_b) = t_{t_a(b)}$ and it is $t_a \circ t_b \circ t_a^{-1}$ as an element of mapping class group.  This relation comes from the choice of Hurwitz system, a set of mutually disjoint arcs except the base point $b_0$ which connecting $b_0$ to $b_i$.

 A choice of generic fiber also gives another equivalence relation. Two monodromy factorizations $W_1$ and $W_2$ are called \emph{simultaneous conjugation equivalence} if $W_2 = f(W_1)$ for some $f\in \mathcal{M}_g$, where $\Sigma_g$ is a generic fiber of the Lefschetz fibration $W_1$.

 Two Lefschetz fibrations $f_1: X_1\to B_1$, $f_2:  X_2\to B_2$ are
 called \emph{isomorphic} if there are orientation preserving diffeomorphisms $H: X_1\to X_2$ and $h:B_1\to B_2$ such that the following diagram commutes:
 \begin{equation}
 \begin{CD}
    X_1  @>H>> X_2 \\
    @V{f_1}VV    @VV{f_2}V \\
    B_1 @>h>> B_2
 \end{CD}
 \end{equation}

 Monodromy factorizations of two isomorphic Lefschetz fibrations are related by a sequence of Hurwitz equivalences and simultaneous conjugation equivalences.

\medskip

\noindent {\em Notation}.
 We denote by $W_1 \cong W_2$ if two monodromy factorizations $W_1$ and $W_2$ are equivalent. In the case that two manifolds $X_1$ and $X_2$ are diffeomorphic, we denote by $X_1 \approx X_2$.

\begin{definition}\label{defn:monodromy_group}
 Let $\pi : X \to S^2$ be a Lefschetz fibration and let $F$ be a fixed generic fiber of the Lefschetz fibration. Let $W = w_n\cdot ... \cdot w_2 \cdot w_1$ be a monodromy factorization of the Lefschetz fibration corresponding to $F$. Then the \emph{monodromy group} $G_F(W)$ is a subgroup of the mapping class group $\mathcal{M}_F = \pi_0(\textrm{Diff}^+(F))$ generated by $w_1, w_2, \cdots, w_n$. We will write $G(W)$ when the generic fiber $F$ is clear from context. The element $w_n \circ \cdots\circ w_2 \circ w_1$ in $\mathcal{M}_F$ is denoted by $\lambda_W$.
\end{definition}

\begin{lemma}\label{lemma:mon-Group}
 If two monodromy factorizations $W_1$ and $W_2$ give isomorphic Lefschetz fibrations over $S^2$, then monodromy groups $G(W_1)$ and $G(W_2)$ are isomorphic as a subgroup of the mapping class group $\mathcal{M}_F$. 
 Moreover if we fix a generic fiber $F$, then $G_F(W_1) = G_F(W_2)$.
\end{lemma}

 A monodromy factorization of a Lefschetz fibration structure on $E(n)_K$ was studied by Fintushel and Stern~\cite{FS:2004} and we could find an explicit monodromy factorization of $E(n)_K$~\cite{Yun:08} with the help of factorization of the identity element in the mapping class group which were discovered by Y. Matsumoto~\cite{Matsumoto:96}, M. Korkmaz~\cite{Korkmaz:2001} and Y. Gurtas~\cite{Gurtas:2004}.

\begin{definition}
\label{definition:MD}
 Let $M(n,g)$ be the desingularization of the double cover of $\Sigma_g
 \times S^2$ branched over $2n (\{pt.\}\times S^2 ) \cup 2 (\Sigma_g 
 \times  \{ pt.\})$.
\end{definition}

\begin{figure}[tbp]
 \includegraphics[scale=1]{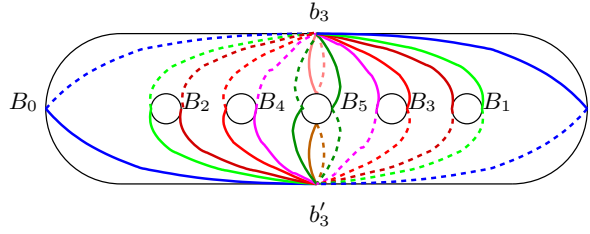}
 \caption{an involution and  its vanishing cycles with $g=2$}\label{fig:generator}
\end{figure}

\begin{lemma}[\cite{Korkmaz:2001,Yun:05}]
\label{lemma:Yun}
 $M(2,g)$ has a monodromy factorization $\eta_{1,g}^2$, where
\[
 \eta_{1,g} = t_{B_0} \cdot t_{B_1} \cdot t_{B_2}\cdot \cdots \cdot t_{B_{2g}} \cdot t_{B_{2g+1}} \cdot t_{b_{g+1}}^2 \cdot t_{b_{g+1}'}^2
\]
 and  $\{B_j, b_{g+1}, b_{g+1}'\}$ are simple closed curves on $\Sigma_{2g+1}$ as in Figure~\ref{fig:generator}.
\end{lemma}

\begin{theorem}[\cite{FS:2004,Yun:08}]
\label{theorem:FS2004}
 Let $K \subset S^3$ be a fibred knot of genus $g$. 
 Then $E(2)_K$, as a genus $(2g+1)$ Lefschetz fibration,
 has a monodromy factorization of the form
\[
 \Phi_K (\eta_{1, g}) \cdot \Phi_K (\eta_{1, g}) \cdot \eta_{1, g} \cdot\eta_{1, g},
\]
 where $\eta_{1, g}^2$ is a monodromy factorization of $M(2,g)$ and
\[
 \Phi_K = \varphi_K\oplus id \oplus id: \Sigma_g \sharp \Sigma_{1} \sharp \Sigma_g \to \Sigma_g \sharp \Sigma_{1} \sharp \Sigma_g
\]
 is a diffeomorphism obtained by using a (geometric) monodromy $\varphi_K$ of $K$ defined by
\[
 S^3 \setminus \nu(K) = (I \times \Sigma_g^1)/ ((1,x) \sim (0, \varphi_K(x))),
\]
 where $\Sigma_g^1$ is an oriented surface of genus $g$ with one boundary component.
\end{theorem}


\bigskip

\section{Isomorphic Lefschetz fibrations}\label{section:isom}

 In this section we construct examples of simply connected isomorphic symplectic Lefschetz fibrations with the same generic fiber but coming 
 from a pair of inequivalent fibred knots.
 In~\cite{FS:2004} Fintushel and Stern constructed families of simply connected symplectic $4$-manifolds with the same Seiberg-Witten invariants. Among them, they considered a set of the following symplectic $4$-manifolds
 \[
  \{ Y(2; K_1, K_2):= E(2)_{K_1} \sharp_{id:\Sigma_{2g+1} \to \Sigma_{2g+1}} E(2)_{K_2} \ |\ K_1, K_2 \ \mathrm{are\ genus\ g\ fibred\ knots} \}
\]
 and they showed that
\[
 \mathcal{SW}_{Y(2;K_1, K_2)} = t_K +  t_K^{-1}.
\]
 In~\cite{Yun:08} we found examples such that $Y(2;K, K_1)$ and $Y(2;K, K_2)$ are diffeomorphic even though $K_1$ is not equivalent to $K_2$. In this section we will generalize such a construction. That is,  we will construct infinitely many pairs $(K,K')$ of inequivalent genus $2$ fibred knots such that all of $Y(2; K, K')$'s are mutually diffeomorphic.

 A family of inequivalent knots with the same Alexander polynomials were constructed by several authors. Among them, Kinoshita and Terasaka~\cite{Kinoshita_Terasaka:57} constructed a nontrivial knot with the
 trivial Alexander polynomial by using an operation, so called knot union. After that, Kanenobu constructed infinitely many inequivalent knots $K_{p,q}$ ($p, q \in \mathbb{Z}$) with the same Alexander polynomials~\cite{Kanenobu:86, Kanenobu:86a}. 
 They constructed the examples from the ribbon fibred knot $4_1\#(-4_1^*)$ by applying the Stallings' twist~\cite{Stallings:78} at two different locations repeatedly, where $K^*$ is the mirror image of $K$.

\begin{figure}[htb]
  \includegraphics{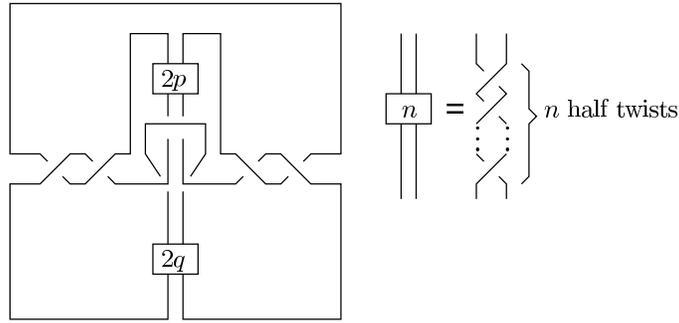}
  \caption{A Kanenobu knot $K_{p,q}$}\label{fig:Kpq}
\end{figure}

 The followings are known to Kanenobu.

\begin{lemma}[\cite{Kanenobu:86}]
\label{lemma:kanenobu} 
 Let $K_{p,q}$ be a Kanenobu knot as in Figure~\ref{fig:Kpq}. Then
\begin{enumerate}
    \item $K_{0,0} = 4_1 \# (- 4_1^*)$
    \item The Alexander matrix of $K_{p,q}$ is $\left(%
    \begin{array}{cc}
        t^2 -3t + 1 & (p-q)t \\
        0 & t^2 - 3t + 1 \\
    \end{array}%
    \right)$
    \item $\Delta_{K_{p,q}}(t) \doteq (t - 3 + t^{-1})^2$
    \item $K_{p,q}$ is a fibred ribbon knot
    \item $K_{p,q} \sim K_{r,s}$ if and only if $(p,q) = (r,s) \text{or} (s,r)$
    \item $K_{p,q}^* \sim K_{-q,-p}$
    \item $K_{p,q}$ is a prime knot if $(p,q) \ne (0,0)$
\end{enumerate}
\end{lemma}

\begin{figure}[htb]
 \centering
 \includegraphics{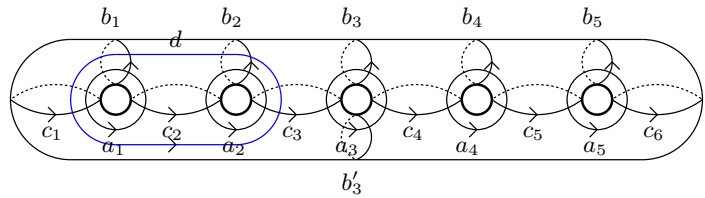}
 \caption{Standard simple closed curves}
 \label{fig:scc}
\end{figure}

 It is not hard to see~\cite{Harer:82} that the monodromy map $\Phi_{K_{p,q}}$ of a Kanenobu knot $K_{p,q}$ is
 \[
   t_d^q \circ t_{c_2}^p\circ t_{a_2} \circ t_{b_2}^{-1} \circ t_{a_1}^{-1} \circ t_{b_1},
 \]
 where $\{a_i, b_i, c_i, d\}$ are simple closed curves in Figure~\ref{fig:scc}.
 Therefore we get that $Y(2;K_{p,q}, K_{r,s})$ has a monodromy factorization of the form
\[
 \Phi_{K_{r,s}}(\eta_{1,2}^2) \cdot \eta_{1,2}^2 \cdot \Phi_{K_{p,q}}(\eta_{1,2}^2) \cdot \eta_{1,2}^2.
\]

\begin{lemma}\label{lemma:key-p}
 For each $k \in \mathbb{Z}_{\ge 0}$ and fixed $p$, $q$, we have following relations:
\begin{eqnarray}
  t_{c_2} &=& \lambda_{t_{c_2}^k(t_{B_2})}(t_{c_2}^{k+1}(t_{B_2})) =
   \lambda_{t_{c_2}^{k+1}(t_{B_3}^{-1})}(t_{c_2}^{k}(t_{B_3})) \label{eqn:key1} \\
  t_d &=&  \lambda_{t_{d}^k(t_{B_4})}(t_{d}^{k+1}(t_{B_4})) =
   \lambda_{t_{d}^{k+1}(t_{B_3})}(t_{d}^{k}(t_{B_3})) \label{eqn:key2}  \\
  t_{c_2} &=& \lambda_{\Phi_{K_{k,q}}(t_{B_3})}(\Phi_{K_{k+1,q}}(t_{B_3})) \label{eqn:key3}\\
  t_d &=& \lambda_{\Phi_{K_{p,k+1}}(t_{B_4})}^{-1}(\Phi_{K_{p,k}}(t_{B_4})). \label{eqn:key4}
\end{eqnarray}
 Therefore
\begin{equation*}
\begin{array}{rclrcl}
 t_{c_2} &\in& G_F(\xi_{p,q} \cdot \xi_{p \pm 1, q}), & t_{c_2} &\in& G_F(\eta_{1,2}^2 \cdot
t_{c_2}(\eta_{1,2}^2)) \\
 t_{d} &\in& G_F(\xi_{p,q} \cdot \xi_{p, q\pm 1}), & t_{d} &\in& G_F(\eta_{1,2}^2 \cdot
t_{d}(\eta_{1,2}^2)).
\end{array}
\end{equation*}

\end{lemma}

\begin{proof}
 It is easy to see the image of a simple closed curve on an oriented surface $\Sigma_5$ under Dehn twists and we get
\begin{eqnarray}
  c_2 &=& (t_{B_2}\circ t_{c_2})(B_2) = (t_{B_3}^{-1}\circ t_{c_2}^{-1})(B_3) \label{eqn:letter1}\\
  d &=& (t_{B_4}\circ t_d)(B_4) = (t_{B_3}\circ t_d)(B_3). \label{eqn:letter2}
\end{eqnarray}

 Equation~(\ref{eqn:key1}) is easily obtained from Equation~(\ref{eqn:letter1}) because
\[
 (t_{B_2}\circ t_{c_2})(t_{B_2}) = t_{c_2} =  (t_{B_3}^{-1}\circ t_{c_2}^{-1})(t_{B_3})
\]
 and, from this, we obtain
\begin{eqnarray*}
 t_{c_2} = t_{t_{c_2}^k(c_2)} &=& \lambda_{t_{c_2}^k}(t_{c_2}) \\
  &=& \lambda_{t_{c_2}^k}( \lambda_{t_{B_2}\circ t_{c_2}}(t_{B_2})) \\
  &=& \lambda_{t_{c_2}^k \circ t_{B_2}\circ t_{c_2}}(t_{B_2}) \\
  &=& \lambda_{t_{c_2}^k\circ t_{B_2}\circ t_{c_2}^{-k} \circ t_{c_2}^{k+1}} (t_{B_2}) \\
  &=& \lambda_{t_{c_2}^k(t_{B_2})} (t_{c_2}^{k+1}(t_{B_2}))
\end{eqnarray*}
 and
\begin{eqnarray*}
 \lambda_{t_{c_2}^{k+1}(t_{B_3}^{-1})}(t_{c_2}^{k}(t_{B_3})) &=& t_{(t_{c_2}^{k+1}(t_{B_3}^{-1})
\circ t_{c_2}^{k})(B_3)} \\
 &=& t_{(t_{c_2}^{k+1} \circ t_{B_3}^{-1} \circ t_{c_2}^{-1})(B_3)} \\
 &=& t_{t_{c_2}^{k+1}(c_2)} = t_{c_2}.
\end{eqnarray*}
 Similarly, we get Equation~(\ref{eqn:key2}) from Equation~(\ref{eqn:letter2}).

 Next we prove Equation~(\ref{eqn:key3}). First, observe that
\[
 (\Phi_{K_{0,0}}(t_{B_3})\circ t_{c_2} \circ \Phi_{K_{0,0}}) (B_3) = c_2
\]
 because $\Phi_{K_{0,0}}(B_3)$ meets with $c_2$ at one point.  Therefore
\begin{eqnarray*}
 \lambda_{\Phi_{K_{k,q}}(t_{B_3})}(\Phi_{K_{k+1,q}}(t_{B_3})) &=&
 t_{(\Phi_{K_{k,q}}(t_{B_3}) \circ \Phi_{K_{k+1,q}})(B_3)} \\
 &=& t_{(t_d^q \circ t_{c_2}^k \circ \Phi_{K_{0,0}}(t_{B_3}) \circ t_{c_2} \circ \Phi_{K_{0,0}})
(B_3)} \\
 &=& t_{(t_d^q \circ t_{c_2}^p)(c_2)} = t_{c_2}.
\end{eqnarray*}

 For equation~(\ref{eqn:key4}), observe that
\[
 (\Phi_{K_{0,0}}(t_{B_4}^{-1})\circ t_{d}^{-1} \circ \Phi_{K_{0,0}})(B_4)=d
\]
 and it implies
\begin{eqnarray*}
 \lambda_{\Phi_{K_{p,k+1}}(t_{B_4})}^{-1}(\Phi_{K_{p,k}}(t_{B_4})) &=&
 t_{(\Phi_{K_{p,k+1}}(t_{B_4}^{-1})\circ \Phi_{K_{p,k}})(B_4)} \\
 &=& t_{(t_d^{k+1} \circ t_{c_2}^p \circ \Phi_{K_{0,0}}(t_{B_4}^{-1}) \circ t_d^{-1} \circ \Phi_{K_{0,0}}) (B_4)} \\
 &=& t_{(t_d^{k+1} \circ t_{c_2}^p) (d)} = t_d.
\end{eqnarray*}
\end{proof}

\begin{lemma}\label{lemma:Y(n)-1}
 For each pair $p, q \in \mathbb{Z}$, we get diffeomorphisms
 \[
   Y(2; K_{p,q}, K_{p+1,q}) \approx Y(2; K_{p+1,q}, K_{p+2,q})
 \]
 and
\[
 Y(2;K_{p,q}, K_{p, q+1}) \approx Y(2; K_{p, q+1}, K_{p, q+2}).
\]
\end{lemma}

\begin{proof}
 $Y(2;K_{p,q}, K_{p+1, q})$ has a monodromy factorization of the form
\[
 \Phi_{K_{p+1,q}}(\eta_{1,2}^2) \cdot \eta_{1,2}^2 \cdot \Phi_{K_{p,q}}(\eta_{1,2}^2) \cdot \eta_{1,2}^2
\]
 where $\Phi_{p,q} = t_d^q \circ t_{c_2}^p\circ t_{a_2} \circ t_{b_2}^{-1} \circ t_{a_1}^{-1} \circ t_{b_1}$.

 By Equation~(\ref{eqn:key3}) in Lemma~\ref{lemma:key-p},
\begin{eqnarray*}
 t_{c_2} &\in& G_F(\Phi_{K_{p+1,q}}(\eta_{1,2}^2) \cdot \Phi_{K_{p,q}}(\eta_{1,2}^2) \\
 t_{c_2} &\in& G_F(t_{c_2}(\eta_{1,2}^2) \cdot \eta_{1,2}^2).
\end{eqnarray*}
 Therefore we get
\begin{eqnarray*}
 \lefteqn{\Phi_{K_{p+1,q}}(\eta_{1,2}^2) \cdot \eta_{1,2}^2 \cdot \Phi_{K_{p,q}}(\eta_{1,2}^2) \cdot \eta_{1,2}^2} \\
 &\sim& \Phi_{K_{p+1,q}}(\eta_{1,2}^2)  \cdot \Phi_{K_{p,q}}(\eta_{1,2}^2) \cdot t_{c_2}(\eta_{1,2}^2) \cdot \eta_{1,2}^2 \\
 &\sim&  \Phi_{K_{0,q}}(\eta_{1,2}^2)  \cdot \Phi_{K_{0,q}}(\eta_{1,2}^2) \cdot t_{c_2}(\eta_{1,2}^2) \cdot \eta_{1,2}^2.
\end{eqnarray*}
 It implies that, for each fixed $q$, $Y(2; K_{p,q}, K_{p+1,q})$ has isomorphic Lefschetz fibration structures, so that they are all diffeomorphic.

 Similarly, by using Equation~(\ref{eqn:key4}) in Lemma~\ref{lemma:key-p}, 
 we can prove that
\[
 Y(2;K_{p,q}, K_{p, q+1}) \approx Y(2; K_{p, q+1}, K_{p, q+2}).
\]
\end{proof}

\begin{theorem}
\label{theorem:Y(n)-1-0-equiv}
 Any two simply connected symplectic $4$-manifolds in
 \[
  \{ Y(2; K_{p,q}, K_{p+1,q}) \ | \ p, q \in \mathbb{Z} \}
\cup
 \{ Y(2;K_{p,q}, K_{p,q+1}) \ | \ p, q \in \mathbb{Z} \}
 \]
 are mutually diffeomorphic.
\end{theorem}

\begin{proof}
 By Lemma~\ref{lemma:Y(n)-1}, the smooth classification problems of simply connected symplecic $4$-manifolds in
 $\{ Y(2; K_{p,q}, K_{r,s}) \ | \ ( r= p \pm 1, s=q) \textrm{ or } (r=p, s= q \pm 1) \}$
 can be reduced to the smooth classification problems of $4$-manifolds in
 $\{ Y(2; K_{p,p}, K_{p+1,p}), Y(2;K_{p,p}, K_{p,p+1}) \ | \ p \in \mathbb{Z} \}$.
 Then, by Lemma~\ref{lemma:kanenobu}, we have $K_{p,p \pm 1} \sim K_{p\pm 1, p}$ and it implies
\begin{eqnarray*}
 Y(2; K_{p,p}, K_{p + 1,p}) &\approx& Y(2;K_{p,p}, K_{p,p + 1}) \\
 Y(2; K_{p,p+1}, K_{p+1,p+1}) &\approx& Y(2;K_{p+1,p}, K_{p+1,p+1}).
\end{eqnarray*}
 Furthermore, by Lemma~\ref{lemma:Y(n)-1}, we also get 
\[
  Y(2; K_{p,p+1}, K_{p+1,p+1}) \approx Y(2;K_{p+1,p+1}, K_{p+2,p+1}).
\] 
 Therefore, for any $p, q \in \mathbb{Z}$, we have 
\[
 Y(2;K_{p,p}, K_{p, p\pm 1}) \approx Y(2;K_{q,q}, K_{q \pm 1, q})
\]
 and it implies the conclusion.
\end{proof}

\bigskip

\section{Nonisomorphic Lefschetz fibrations}
\label{section:nonisom}

 In this section we investigate some algebraic and graph theoretic properties of $\xi_{p,q} = \Phi_{K_{p,q}}(\eta_{1,2}^2) \cdot \eta_{1,2}^2$ and its monodromy group $G_{\Sigma_5}(\xi_{p,q})$.
 In~\cite{Humphries:79}, Humphries showed that the minimal number of Dehn twist generators of the mapping class group $\mathcal{M}_g$ or $\mathcal{M}_g^1$ is $2g+1$ by using symplectic transvection and modulo two Euler number of a graph.

\begin{definition}
\label{definition:graph}
 Let $\{\gamma_1, \gamma_2, \cdots, \gamma_{2g}\}$ be a set of simple closed curves on $\Sigma_{g}$ which generate $H_1(\Sigma_g;\mathbb{Z}_2)$. 
 Let $\Gamma(\gamma_1, \gamma_2, \cdots, \gamma_{2g})$ be a graph which is defined by
\begin{itemize}
 \item a vertex for each simple closed curve $\gamma_i$, $i=1, 2, \cdots, 2g$
 \item an edge between $\gamma_i$ and $\gamma_j$ if $i(\gamma_i, \gamma_j) = 1$ (mod $2$) where $i(\gamma_i, \gamma_j)$ is the minimum number of intersection between two simple closed curves $\gamma_i$ and $\gamma_j$
 \item no intersections between any two edges.
\end{itemize}
 Let $\gamma$ be a simple closed curve on $\Sigma_5$, then $\gamma =\sum_{i=1}^{2g} \varepsilon_i \gamma_i$ ($\varepsilon_i = 0 \textrm{ \ or \ } 1$) as an element of $H_1(\Sigma_g; \mathbb{Z}_2)$.
 Let $\overline{\gamma} := \cup_{\varepsilon_i = 1} \overline{\gamma_i}$  where $\overline{\gamma_i}$ be the union of all closure of half edges with one end vertex $\gamma_i$.
 We define $\chi_\Gamma(\gamma)$ as the modulo $2$ Euler number $\chi_\Gamma(\overline{\gamma})$.
\end{definition}

\begin{lemma}
\label{lemma:chi=1}
 Let $\Gamma(\gamma_1, \cdots, \gamma_{2g})$ be a graph of simple closed curves $\{\gamma_1, \cdots, \gamma_{2g}\}$ which generate $\mathbb{Z}_2$-vector space $H_1(\Sigma_g; \mathbb{Z}_2)$.
 Let $G_{\Gamma,g}$ be a subgroup of $\mathcal{M}_g$ generated by
\[
 \{t_\alpha \ | \   \alpha \textrm{ is a nonseparating simple closed curve on } \Sigma_g \textrm{ such that } \chi_\Gamma(\alpha)=1 \}.
\]
 Then $G_{\Gamma, g}$ is a nontrivial proper subgroup of $\mathcal{M}_g$. Moreover, if $\beta$ is a non-separating simple closed curve on $\Sigma_g$ with $\chi_\Gamma (\beta) = 0$, then $t_\beta \not\in G_{\Gamma, g}$.
\end{lemma}

\begin{proof}
 Let us prove that $G_{\Gamma, g}$ is a nontrivial proper subgroup of $\mathcal{M}_g$. Each element in $H_1(\Sigma_g; \mathbb{Z}_2)$ can be represented by a non-separating simple closed curve on $\Sigma_g$ and $\mathcal{M}_g$ acts transitively on the set of all non-separating simple closed curves on $\Sigma_g$.

 If $c$ is a non-separating simple closed curve on $\Sigma_g$ such that $\chi_\Gamma(c) =1$, then
\begin{equation*}
 t_c(\gamma) = \begin{cases}
                \gamma, & \textrm{\ if \ } i(c, \gamma) = 0 (\textrm{\ mod\ } 2) \\
                c + \gamma, & \textrm{\ if \ } i(c, \gamma) = 1 (\textrm{\ mod\ } 2)
               \end{cases}
\end{equation*}
 and, for the $i(c, \gamma) = 1 \pmod 2$ case,
\[
 \overline{t_c(\gamma)} = \bar{c} \cup \bar{\gamma} \textrm{\ and \ } \ \bar{c} \cap \bar{\gamma} =
\textrm{ odd number of points }.
\]
 So
\begin{equation*}
 \chi_\Gamma( t_c(\gamma) ) =
 \begin{cases}
  \chi_\Gamma(\gamma), & \textrm{\ if \ } i(c, \gamma) = 0 \pmod 2 \\
  \chi_\Gamma(\bar{\gamma}) + \chi_\Gamma(\bar{c}) + \textrm{ odd number}
   \equiv \chi_\Gamma(\gamma),  & \textrm{\ if \ } i(c, \gamma) = 1 \pmod 2.
 \end{cases}
\end{equation*}
 For any $f \in G_{\Gamma, g}$, 
 $f$ is of the form $t_{c_k}^{\epsilon_k} \circ t_{c_{k-1}}^{\epsilon_{k-1}} \circ \cdots \circ t_{c_2}^{\epsilon_2} \circ t_{c_1}^{\epsilon_1}$,
 where each $c_i$ is a non-separating simple closed curve with $\chi_\Gamma(c_i) =1$ and it implies $\chi_\gamma(f(\gamma)) \equiv \chi_\Gamma(\gamma) \pmod 2$.
 Therefore if $G_{\Gamma, g} = \mathcal{M}_g$, then for any non-separating simple closed curves $\gamma$ on $\Sigma_g$ we have to have $\chi_\Gamma(\gamma) = 1$. It is clearly impossible.
 So $G_{\Gamma,g}$ is a nontrivial proper subgroup of $\mathcal{M}_g$.

 If $\beta$ is a non-separating simple closed curve with $\chi_\Gamma(\beta) = 0$, then, for simple closed curve $\gamma$ on $\Sigma_g$ with $i(\beta,\gamma) =1$, we have $\chi_\Gamma(t_\beta(\gamma)) \not\equiv \chi_\Gamma(\gamma) \pmod 2$. Therefore $t_\beta \not\in G_{\Gamma, g}$.
\end{proof}

\begin{lemma}
\label{lemma:pq}
 For any $p, q \in \mathbb{Z}$,
\[
 t_{c_2}, \ t_d \notin G_F(\xi_{p,q}) = G_F(\Phi_{K_{p,q}}(\eta_{1,2}^2) \cdot \eta_{1,2}^2).
\]
\end{lemma}

\begin{proof}
 We will prove this in four cases.

\medskip

\noindent 
 \textbf{Case 1: $p$ and $q$ are even integers:}
 Let us consider
\[
 \Gamma_1 = \Gamma(\{c_1, a_1, a_2, b_2, a_3, b_3,  a_4, a_5, B_2, B_4\})
\]
 where $\{a_i, b_i, c_i, d_i, B_i\}$  are simple closed curves on $\Sigma_5$ as in Figure~\ref{fig:generator} and in Figure~\ref{fig:scc}.
 Then graph $\Gamma_1$ is given as in Figure~\ref{fig:graph1} and we have following relations in $H_1(\Sigma_5;\mathbb{Z}_2)$:
\begin{eqnarray}
 b_2 &=& c_1 + c_2, \nonumber \\
 b_3 &=& c_1 + c_2 + c_3 = b_3',\nonumber \\
 B_0 &=& a_1 + a_2 + a_3 + a_4 + a_5,\nonumber \\
 B_1 &=& B_2 + a_1 + a_5, \label{eqn:bi} \\
 B_2 &=& B_4 + a_2 + a_4 + c_2 + c_5, \nonumber\\
 B_3 &=& B_4 + a_2 + a_4,\nonumber \\
 B_4 &=& a_3 + c_3  + c_4,\nonumber \\
 B_5 &=& a_3 + b_3 + b_3' = a_3.\nonumber
\end{eqnarray}

\begin{figure}[htb]
 \centering
 \includegraphics{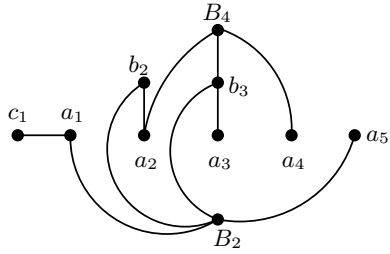}
 \caption{Graph $\Gamma_1$}
 \label{fig:graph1}
\end{figure}

 Therefore we get
\begin{eqnarray*}
 c_2 &=& b_2 + c_1, \\
 c_3 &=& b_2 + b_3, \\
 c_4 &=& B_4 + a_3 + b_2 + b_3, \\
 c_5 &=& B_2 + B_4 + a_2 + a_4 + b_2 + c_1, \\
 c_6 &=& c_1 + c_2 + c_3 + c_4 + c_5 = c_1 + a_2 + a_3 + a_4 + B_2
\end{eqnarray*}
 and it implies that each $c_i$ is in the space spanned by 
 \[\{c_1, a_1, a_2, b_2, a_3, b_3,  a_4, a_5, B_2, B_4 \}.\]
 Since $b_i = c_1 + \cdots + c_i$, each $b_i$ is also in the space spanned by
 \[\{c_1, a_1, a_2, b_2, a_3, b_3,  a_4, a_5, B_2, B_4 \}\] and therefore
 $\{c_1, a_1, a_2, b_2, a_3, b_3,  a_4, a_5, B_2, B_4 \}$ generates $H_1(\Sigma_5; \mathbb{Z}_2)$.

 Furthermore, by drawing figures, we can easily obtain the following relations:
\begin{eqnarray*}
 \Phi_{K_{0,0}}(B_5) &=& B_5, \\
 \Phi_{K_{0,0}}(B_4) &=& B_4 + a_2, \\
 \Phi_{K_{0,0}}(B_3) &=& B_3 + b_2 = B_4 + a_2 + a_4 + b_2, \\
 \Phi_{K_{0,0}}(B_2) &=& B_2 + a_1 + b_2 + a_2, \\
 \Phi_{K_{0,0}}(B_1) &=& B_1 + b_1 + b_2 + a_2 = B_2 + a_1+ a_2 + a_5 + b_1 + b_2, \\
 \Phi_{K_{0,0}}(B_0) &=& B_0 + a_1 + b_1 + a_2 + b_2 = a_3 + a_4 + a_5 + b_1 + b_2.
\end{eqnarray*}

 Therefore
\[
 \chi_{\Gamma_1}(a_i) = \chi_{\Gamma_1}(B_i) = \chi_{\Gamma_1}(\Phi_{K_{0,0}}(B_i)) = 1,  \textrm{\ for \ } i=0,1,2,3,4,5
\]
 and $\chi_{\Gamma_1}(c_1) = \chi_{\Gamma_1}(c_6) = 1$. 
 So we have
\[
 \{t_{B_i}, \Phi_{K_{0,0}}(t_{B_i}), t_{a_j}, t_{b_3}, t_{b_3'},  t_{c_1},  t_{c_6} \ | \ i= 0, 1, 2, 3, 4, 5,\ j=1,2,3,4,5 \}
  \subset G_{\Gamma_1, 5}
\]
 and each generator of the group $G_F(\Phi_{K_{0,0}}(\eta_{1,2}^2)\cdot \eta_{1,2}^2)$ is an element of $G_{\Gamma_1, 5}$. It implies that $G_F(\Phi_{K_{0,0}}(\eta_{1,2}^2)\cdot \eta_{1,2}^2) \le G_{\Gamma_1, 5}$.

 But we have
\[
 \chi_{\Gamma_1}(c_j) = \chi_{\Gamma_1}(d) = 0
\]
 for $j=2,3,4, 5$ and therefore
\[
  t_{c_2},\ t_{c_3},\  t_{c_4},\ t_{c_5},\ t_d \notin G_{\Gamma_1,5}.
\]
 It implies that 
 $t_{c_2}, \ t_d \notin G_F(\Phi_{K_{0,0}}(\eta_{1,2}^2)\cdot \eta_{1,2}^2)$.

 Since $\mathbb{Z}_2$-homology class of $\Phi_{K_{2p, 2q}} (B_i)$ and $\Phi_{K_{0, 0}} (B_i)$ are the
 same for any $p, q \in \mathbb{Z}$, we get
\[
 \chi_{\Gamma_1}(\Phi_{K_{2p, 2q}} (B_i)) = \chi_{\Gamma_1}(\Phi_{K_{0,0}} (B_i))
\]
 for $i=0, 1, 2, 3, 4, 5$.
 It implies that $G_F(\Phi_{K_{2p,2q}}(\eta_{1,2}^2)\cdot \eta_{1,2}^2) \le G_{\Gamma_1, 5}$ and therefore we have
 $t_{c_2}, \ t_d \notin G_F(\Phi_{K_{2p,2q}}(\eta_{1,2}^2)\cdot \eta_{1,2}^2)$.

\medskip

\noindent
 \textbf{Case 2: $p$ is an odd and $q$ is an even integer:}
 Let us consider
\[
 \Gamma_2 = \Gamma(\{a_3, b_3, B_1, B_2, B_3, B_4, d_1, d_2, d_3, d_4\})
\]
 where $\{a_i, b_i, c_i, d_i, B_i\}$ are simple closed curves on $\Sigma_5$ as in Figure~\ref{fig:generator}, Figure~\ref{fig:scc} and Figure~\ref{fig:d}. Then the graph $\Gamma_2$ is as in Figure~\ref{fig:graph2}.

\begin{figure}[htb]
 \centering
 \includegraphics{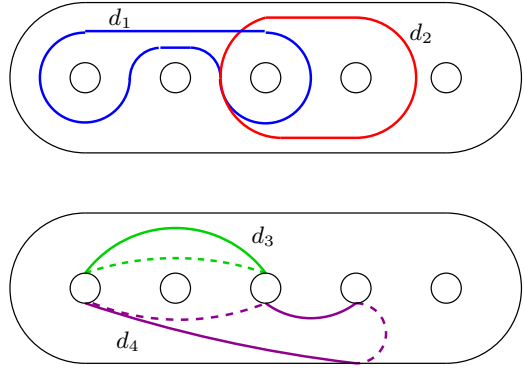}
 \caption{Simple closed curves $d_i$}
 \label{fig:d}
\end{figure}

\begin{figure}[htb]
 \centering
 \includegraphics{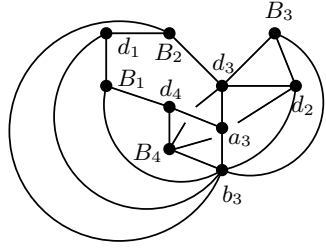}
 \caption{Graph $\Gamma_2$}
 \label{fig:graph2}
\end{figure}

 Since the equations in (\ref{eqn:bi}) are still valid and we have
\begin{eqnarray*}
 d_1 &=& a_1 + a_3, \\
 d_2 &=& a_3 + a_4, \\
 d_3 &=& c_2  + c_3, \\
 d_4 &=&  c_2 + c_3 + c_5 + c_6, \\
 b_3 &=& c_1 + c_2 + c_3 = c_4 + c_5 + c_6 ,
\end{eqnarray*}
 we get
\begin{eqnarray}
 a_1 &=& d_1 + a_3, \nonumber\\
 a_2 &=& B_3 + B_4 + a_3 + d_2, \nonumber \\
 a_3 &=& B_5,  \nonumber \\
 a_4 &=& a_3 + d_2, \nonumber\\
 a_5 &=& B_1 + B_2 + a_3 + d_1, \label{eqn:ai-2} \\
 c_1 &=& b_3 + d_3, \nonumber \\
 c_2 &=& a_3 + b_3 + d_4 + B_4, \nonumber \\
 c_3 &=& a_3 + b_3 + d_3 + d_4 + B_4, \nonumber\\
 c_4 &=& b_3 + d_3 + d_4, \nonumber \\
 c_5 &=& a_3 + b_3 + c_2 + d_4 + B_2 + B_3 + B_4. \nonumber
\end{eqnarray}

 Hence Equations~(\ref{eqn:ai-2}) implies that
 $\{a_3, b_3, B_1, B_2, B_3, B_4, d_1, d_2, d_3, d_4 \}$
 is a basis of $\mathbb{Z}_2$-vector space $H_1(\Sigma_5;\mathbb{Z}_2)$.

 Since $\Phi_{K_{1,0}} = t_{c_2} \circ t_{a_2} \circ t_{b_2}^{-1} \circ t_{a_1}^{-1} \circ t_{b_1}$, 
 using (\ref{eqn:ai-2}) and $b_i = c_1 +  \cdots + c_i$, we get
\begin{eqnarray}
 \Phi_{K_{1,0}}(B_0) &=& B_0 + a_1 + b_1 + a_2 + b_2, \nonumber \\
 \Phi_{K_{1,0}}(B_1) &=& B_1 + b_1 + b_2 + a_2  + c_2 \nonumber \\
 &=& B_1 + B_3 + B_4 + a_3 + d_2 , \nonumber \\
 \Phi_{K_{1,0}}(B_2) &=& B_2 + a_1 + b_2 + a_2 + c_2 \label{eqn:Bi-2} \\
 &=& B_2 + B_3 + B_4 + b_3 + d_1 + d_2 + d_3,  \nonumber \\
 \Phi_{K_{1,0}}(B_3) &=& B_3 + b_2 + c_2 = B_3 + b_3 + d_3, \nonumber \\
 \Phi_{K_{1,0}}(B_4) &=& B_4 + a_2 + c_2 = B_3 + B_4 + b_3 + d_2 + d_4 , \nonumber \\
 \Phi_{K_{1,0}}(B_5) &=& B_5. \nonumber
\end{eqnarray}
  
 A computation of  $\chi_{\Gamma_2}$ shows that
\begin{equation}
 \chi_{\Gamma_2}(B_i) = \chi_{\Gamma_2}(\Phi_{K_{1,0}}(B_i)) = \chi_{\Gamma_2}(b_3) =
\chi_{\Gamma_2}(b_3') = \chi_{\Gamma_2}(a_3) =1
\end{equation}
 for each $i=0,1,2,3,4,5$ and
\begin{equation}
 \chi_{\Gamma_2}(c_1) = \chi_{\Gamma_2}(c_2) = \chi_{\Gamma_2}(a_1) = \chi_{\Gamma_2}(a_2) = \chi_{\Gamma_2}(b_2) = \chi_{\Gamma_2}(d) = 0.
\end{equation}
 Therefore $G_F( \Phi_{K_{1,0}}(\eta_{1,2}^2) \cdot \eta_{1,2}^2) \le G_{\Gamma_2, 5}$ and, since $t_{c_2},\  t_{d} \notin G_{\Gamma_2, 5}$, 
 we get
\[
 t_{c_2},\  t_{d} \notin G_F( \Phi_{K_{1,0}}(\eta_{1,2}^2) \cdot \eta_{1,2}^2).
\]

 Furthermore, since $\Phi_{K_{2p + 1, 2q}}(B_i)$ and $\Phi_{K_{1,0}}(B_i)$ represent the same element in $H_1(\Sigma_2; \mathbb{Z}_2)$, we get  $\chi_{\Gamma_2}(\Phi_{K_{2p + 1, 2q}}(B_i)) = \chi_{\Gamma_2}(\Phi_{K_{1,0}}(B_i)) = 1$ and it implies that 
\[
 t_{c_2},\  t_{d} \notin G_F( \Phi_{K_{2p + 1, 2q}}(\eta_{1,2}^2) \cdot \eta_{1,2}^2)
\]
 for any $p,\ q \in \mathbb{Z}$ because $G_F( \Phi_{K_{2p + 1, 2q}}(\eta_{1,2}^2) \cdot \eta_{1,2}^2) \le G_{\Gamma_2, 5}$.

\medskip

\noindent
 \textbf{Case 3: $p$ is an even and $q$ is an odd integer:}
 We want to find a graph
\[
 \Gamma_3 =\Gamma(\{ \gamma_1, \gamma_2, \cdots, \gamma_{10} \})
\]
 satisfying
\begin{equation}\label{eqn:graph3-1}
 \chi_{\Gamma_3}(B_i) = \chi_{\Gamma_3}(\Phi_{K_{0,1}}(B_i)) = \chi_{\Gamma_3}(b_3) = \chi_{\Gamma_3}(b_3') = \chi_{\Gamma_3}(a_3) =1
\end{equation}
 for $i=0, 1, 2, 3, 4, 5$ and
\begin{equation}\label{eqn:graph3-2}
 \chi_{\Gamma_3}(c_2) = \chi_{\Gamma_3}(d) = 0.
\end{equation}
 Note that we observe the following relations in $H_1(\Sigma_5; \mathbb{Z}_2)$:
 
\begin{equation}
\label{eqn:K01}
 \begin{array}{c|c|c}
 \hline
   & \Phi_{K_{0,0}}(B_i) & \Phi_{K_{0,1}}(B_i) \\ \hline
  B_0 & B_0 +  a_1 + b_1 + a_2 + b_2 & B_0 + a_1 + b_1 + a_2 + b_2 \\ \hline
  B_1 & B_1 + b_1 + b_2 + a_2 & B_1 + b_1 + a_2 + b_2 + d \\ \hline
  B_2 & B_2 + a_1 + b_2 + a_2 & B_2 + a_1 + b_2 + a_2 \\ \hline
  B_3 & B_3 + b_2 & B_3 + b_2 \\ \hline
  B_4 & B_4 + a_2 & B_4 + a_2 + d  \\ \hline
  B_5 & B_5 & B_5  \\ \hline
 \end{array}
\end{equation}

\medskip

 Hence, by Lemma~\ref{lemma:chi=1} and Equations (\ref{eqn:graph3-1}), (\ref{eqn:graph3-2}) and (\ref{eqn:K01}), 
 we need to find a graph $\Gamma_3$ such that
\begin{itemize}
 \item even number of $\{ a_1, b_1, a_2, b_2\}$ have $\chi_{\Gamma_3} = 0$
 \item odd number of $\{ b_1, b_2, a_2 \}$ have $\chi_{\Gamma_3} = 0$
 \item even number of $\{ a_1, b_2, a_2\}$ have $\chi_{\Gamma_3} = 0$
 \item even number of $\{ b_2\}$ have $\chi_{\Gamma_3} = 0$
 \item odd number of $\{ a_2 \}$ have $\chi_{\Gamma_3} = 0$.
\end{itemize}

 Therefore $\{ B_1, B_2, B_3, B_4, b_1, b_2, b_3, a_3 \}$
 might be a subset of $G_{\Gamma_3, 5}$ and we will extend it to a basis of $H_1(\Sigma_5; \mathbb{Z}_2)$ by adding two simple closed curves $d_1$, $d_2$ as in Figure~\ref{fig:d} so that
\[
 \Gamma_3 = \Gamma(\{ B_1, B_2, B_3, B_4, b_1, b_2, b_3, a_3, d_1, d_2 \}).
\]
 Then $\Gamma_3$ is the graph as in Figure~\ref{fig:graph3} and it satisfies Equations~(\ref{eqn:graph3-1}) and (\ref{eqn:graph3-2}).

\begin{figure}[htb]
 \centering
 \includegraphics{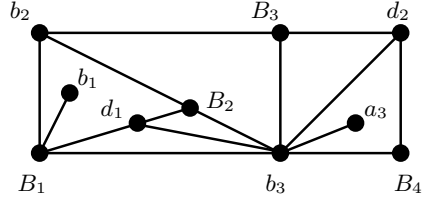}
 \caption{Graph $\Gamma_3$}
 \label{fig:graph3}
\end{figure}

 Therefore $G_F( \Phi_{K_{0,1}}(\eta_{1,2}^2) \cdot \eta_{1,2}^2) \le G_{\Gamma_3, 5}$ and, since $t_{c_2},\ t_{d} \notin G_{\Gamma_3, 5}$, 
 we get
\[
 t_{c_2},\  t_{d} \notin G_F( \Phi_{K_{0,1}}(\eta_{1,2}^2) \cdot \eta_{1,2}^2)
\]
 and
\[
 t_{c_2},\  t_{d} \notin G_F( \Phi_{K_{2p,2q+1}}(\eta_{1,2}^2) \cdot \eta_{1,2}^2)
\]
 for any $p, q \in \mathbb{Z}$.

\medskip

\noindent
 \textbf{Case 4: $p$ and $q$ are odd integers:}
 We want to find a graph
\[
 \Gamma_4 =\Gamma(\{ \gamma_1, \gamma_2, \cdots, \gamma_{10} \})
\]
 satisfying
\begin{equation}\label{eqn:graph4-1}
 \chi_{\Gamma_4}(B_i) = \chi_{\Gamma_4}(\Phi_{K_{1,1}}(B_i)) = \chi_{\Gamma_4}(b_3) = \chi_{\Gamma_4}(b_3') = \chi_{\Gamma_4}(a_3) =1
\end{equation}
 for $i=0, 1, 2, 3, 4, 5$ and
\begin{equation}\label{eqn:graph4-2}
 \chi_{\Gamma_4}(c_2) = \chi_{\Gamma_4}(d) = 0.
\end{equation}
 Note that we observe the following relations  in $H_1(\Sigma_5; \mathbb{Z}_2)$:
 
\begin{equation}\label{eqn:K11}
 \begin{array}{c|c|c}
 \hline
   & \Phi_{K_{0,0}}(B_i) & \Phi_{K_{1,1}}(B_i) \\ \hline
  B_0 & B_0 +  a_1 + b_1 + a_2 + b_2 & B_0 + a_1 + b_1 + a_2 + b_2 \\ \hline
  B_1 & B_1 + b_1 + b_2 + a_2 & B_1  + b_1 + a_2 + b_2 + c_2 + d \\ \hline
  B_2 & B_2 + a_1 + b_2 + a_2 & B_2 + a_1 + b_2 + a_2 + c_2 \\ \hline
  B_3 & B_3 + b_2 & B_3 + b_2 + c_2 \\ \hline
  B_4 & B_4 + a_2 & B_4 + a_2 + c_2 + d  \\ \hline
  B_5 & B_5 & B_5 \\ \hline
 \end{array}
\end{equation}

\medskip

 Hence, by Lemma~\ref{lemma:chi=1} and Equations~(\ref{eqn:graph4-1}), (\ref{eqn:graph4-2}) and (\ref{eqn:K11}), 
 we need to find a graph $\Gamma_4$ such that
\begin{itemize}
 \item even number of $\{ a_1, b_1, a_2, b_2\}$ have $\chi_{\Gamma_4} = 0$
 \item even number of $\{ b_1, b_2, a_2 \}$ have $\chi_{\Gamma_4} = 0$
 \item odd number of $\{ a_1, b_2, a_2\}$ have $\chi_{\Gamma_4} = 0$
 \item odd number of $\{ b_2\}$ have $\chi_{\Gamma_4} = 0$
 \item even number of $\{ a_2 \}$ have $\chi_{\Gamma_4} = 0$.
\end{itemize}

Therefore $\{ B_1, B_2, B_3, B_4, a_1, a_2, b_3, a_3 \}$
 might be a subset of  $G_{\Gamma_4, 5}$ and we will extend it to a basis of $H_1(\Sigma_5; \mathbb{Z}_2)$ by adding two simple closed curves $d_3$, $d_4$ as in Figure~\ref{fig:d} so that
\[
 \Gamma_4 = \Gamma(\{ B_1, B_2, B_3, B_4, a_1, a_2, a_3, b_3, d_3, d_4 \}).
\]

 \begin{figure}[htb]
 \centering
 \includegraphics{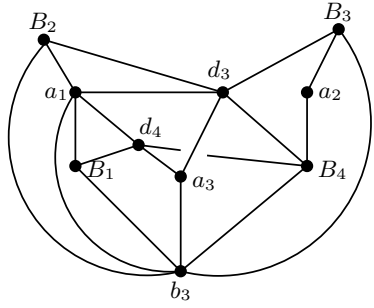}
 \caption{Graph $\Gamma_4$}
 \label{fig:graph4}
\end{figure}

 Then $\Gamma_4$ is the graph as in Figure~\ref{fig:graph4} and it satisfies Equations~(\ref{eqn:graph4-1}) and (\ref{eqn:graph4-2}).
 Therefore $G_F( \Phi_{K_{1,1}}(\eta_{1,2}^2) \cdot \eta_{1,2}^2) \le G_{\Gamma_4, 5}$ and since $t_{c_2},\ t_{d} \notin G_{\Gamma_4, 5}$, 
 we get
\[
 t_{c_2},\  t_{d} \notin G_F( \Phi_{K_{1,1}}(\eta_{1,2}^2) \cdot \eta_{1,2}^2)
\]
 and
\[
 t_{c_2},\  t_{d} \notin G_F( \Phi_{K_{2p+1,2q+1}}(\eta_{1,2}^2) \cdot \eta_{1,2}^2)
\]
 for any $p, q \in \mathbb{Z}$.
\end{proof}

\begin{remark}
 We can doubly check the above statements by using representation of mapping class group in a symplectic group. The following is suggested by S. Humphries~\cite{Humphries:09}: There is a natural map
\[
 \psi_n: \mathcal{M}_5 \xrightarrow{\psi} Sp(10, \mathbb{Z}) \xrightarrow{q_n} Sp(10, \mathbb{Z}/ n \mathbb{Z})
\]
 where, for each $t_\gamma \in \mathcal{M}_5$,
\[
 \psi(t_\gamma) : H_1(\Sigma_5, \mathbb{Z}) \to H_1(\Sigma_5, \mathbb{Z})
\]
 is an integral matrix representation of the mapping class group action on the integral first homology group. After that, we reduce the coefficient of the symplectic group to $\mathbb{Z}/n\mathbb{Z}$ by taking a quotient map $q_n$. It is easy to check that
\[
 \psi_2(t_{c_2}^2) \in \psi_2(G_F(\xi_{p,q}))\ \mathrm{for \ any} \ 
 (p,q) \in \mathbb{Z}^2,
\]
and it implies
\[
 \psi_2(G_F(\xi_{p,q})) = \psi_2(G_F(\xi_{r,s})) \textrm{ \ \ if \ \ } 
 (p,q) \equiv (r,s) \pmod 2.
\]
 An explicit group order computation by using a computer algebra system
 such as GAP~\cite{GAP4} or SAGEMATH~\cite{SAGE} shows that
\begin{eqnarray*}
 \textrm{Order}(\psi_2(G_F(\xi_{p,q}))) &=& 50030759116800, \\
 \textrm{Order}(\langle \psi_2(G_F(\xi_{p,q})\cup \{ t_{c_2}\}) \rangle) &=& 24815256521932800, \\
 \textrm{Order}(\langle \psi_2(G_F(\xi_{p,q})\cup \{ t_d\} ) \rangle ) &=& 24815256521932800, \\
 \textrm{Order}(\psi_2(\mathcal{M}_5)) &=& 24815256521932800,
\end{eqnarray*}
 and it implies
\[
 t_{c_2}, t_d \not\in G_F(\xi_{p,q}) \textrm{ \ for any \ } p,q \in \mathbb{Z}.
\]
\end{remark}

\begin{theorem}
\label{theorem:pq}
 $\xi_{p,q}$ is not equivalent to $\xi_{r,s}$ if $(p,q) \not\equiv (r,s)
  \pmod 2$.
\end{theorem}

\begin{proof}
 We obtain the following table from the proof of Lemma~\ref{lemma:pq}:
 
\begin{equation*}
 \begin{array}{c|c}
 \hline
   & G_{\Gamma_i,5} \textrm{ \ does not contain \ } \\ \hline
  \Gamma_1 &
  t_{\Phi_{K_{1,0}}(B_j)}, (j=1,2,3,4),
  t_{\Phi_{K_{0,1}}(B_1)}, t_{\Phi_{K_{0,1}}(B_4)}, t_{\Phi_{K_{1,1}}(B_2)},
  t_{\Phi_{K_{1,1}}(B_3)}\\ \hline
  \Gamma_2 & t_{\Phi_{K_{0,0}}(B_j)}, (j=1,2,3,4),
  t_{\Phi_{K_{0,1}}(B_2)}, t_{\Phi_{K_{0,1}}(B_3)}, t_{\Phi_{K_{1,1}}(B_1)},
  t_{\Phi_{K_{1,1}}(B_4)}\\ \hline
  \Gamma_3 & t_{\Phi_{K_{0,0}}(B_1)}, t_{\Phi_{K_{0,0}}(B_4)}, t_{\Phi_{K_{1,0}}(B_2)},
  t_{\Phi_{K_{1,0}}(B_3)}, t_{\Phi_{K_{1,1}}(B_j)}, (j=1,2,3,4)
   \\ \hline
  \Gamma_4 & t_{\Phi_{K_{0,0}}(B_2)}, t_{\Phi_{K_{0,0}}(B_3)}, t_{\Phi_{K_{1,0}}(B_1)},
  t_{\Phi_{K_{1,0}}(B_4)}, t_{\Phi_{K_{0,1}}(B_j)}, (j=1,2,3,4)  \\ \hline
 \end{array}
\end{equation*}

\medskip

 Therefore we get the following statement: 
 $t_{\Phi_{K_{p,q}}(B_j)}$ is not contained in $G_{\Gamma_i, 5}$ if and only if $t_{\Phi_{K_{\epsilon_p,\epsilon_q}}(B_j)}$ is not contained in $G_{\Gamma_i, 5}$, where $\epsilon_p, \epsilon_q \in \{0, 1\}$ and $p \equiv \epsilon_p$, $q \equiv\epsilon_q \pmod 2$.
 It implies that
\[
 \xi_{p,q} \not\sim \xi_{r,s} \textrm{ \ if \ } (p,q) \not\equiv (r, s) \pmod 2.
\]

 For example, if $(p,q) \equiv (0,0)$ and $(r,s) \equiv (1,0) \pmod 2$, then
\[
 t_{\Phi_{K_{p,q}}(B_j)} \not\in G_{\Gamma_2, 5},\ (j=1,2,3,4)
\]
 and $G_F(\xi_{r,s}) \leq G_{\Gamma_2, 5}$. Therefore $t_{\Phi_{K_{p,q}}(B_j)} \in G_F(\xi_{p,q})$ but $t_{\Phi_{K_{p,q}}(B_j)} \not\in G_F(\xi_{r,s})$ for $j=1,2,3,4$. It implies $G_F(\xi_{p,q}) \ne G_F(\xi_{r,s})$ and $\xi_{p,q} \not\sim \xi_{r,s}$.
\end{proof}

\begin{corollary}
 If $p \not\equiv q \pmod 2$, then the knot surgery $4$-manifold $E(2)_{K_{p,q}}$ has at least $2$ nonisomorphic genus $5$ Lefschetz fibration structures.
\end{corollary}

\begin{proof}
 It follows from Lemma~\ref{lemma:kanenobu}.
 Since $K_{p,q}$ is equivalent to $K_{q,p}$, we get a diffeomorphism $E(n)_{K_{p,q}} \approx E(n)_{K_{q,p}}$. Because $S^1 \times (S^3 \setminus \nu(K)) \approx S^1 \times (S^3 \setminus \nu(K^*))$, we also have $E(n)_K \approx E(n)_{K^*}$.
 So we get that $E(n)_{K_{p,q}} \approx E(n)_{K_{p,q}^*} \approx E(n)_{K_{-q,-p}}$. The last diffeomorphism is $E(n)_{K_{p,q}} \approx E(n)_{K_{q,p}} \approx E(n)_{K_{q,p}^*}\approx E(n)_{K_{-p,-q}}$.
 Therefore we have diffeomorphisms
\[
 E(2)_{K_{p,q}} \approx E(2)_{K_{q,p}} \approx E(2)_{K_{-p, -q}} \approx E(2)_{K_{-q, -p}}
\]
 and 
 $\{ \xi_{p,q}, \xi_{q,p}, \xi_{-p,-q}, \xi_{-q, -p} \}$
 are Lefschetz fibration structures which have diffeomorphic underlying $4$-manifolds.
 But, by Theorem~\ref{theorem:pq}, we know that $\xi_{p,q} \not\sim \xi_{q,p}$ because $(p,q) \not\equiv (q,p) \pmod 2$. Therefore we get a conclusion.
\end{proof}

\begin{remark}
 At this time we do not know how to distinguish each elements in
\[
 \{ \ \xi_{p,q} \ | \ p, q \in \mathbb{Z} \ \}
\]
 up to Lefschetz fibration isomorphism.
\end{remark}

\begin{remark}
 We are interested in the question whether the knot surgery $4$-manifold $E(2)_K$ admits infinitely many nonisomorphic Lefschetz fibrations over $S^2$ with the same generic fiber. In Theorem~\ref{theorem:Y(n)-1-0-equiv} we constructed a family of simply connected genus $5$ Lefschetz fibrations over $S^2$ whose underlying space are all diffeomorphic and they are constructed from a pair of inequivalent prime fibred knots. We expect that they are strong candidates for admitting infinitely many nonisomorphic Lefschetz fibrations. We leave this problem for future research project.
\end{remark}

\end{document}